\documentclass{amsart}

\usepackage{mathtools}
\usepackage{enumitem}
\usepackage{hyperref}
\usepackage{cleveref}

\newtheorem{thm}{Theorem}
\newtheorem{lem}[thm]{Lemma}
\newtheorem{cor}[thm]{Corollary}
\newtheorem{obs}[thm]{Observation}

\theoremstyle{definition}
\newtheorem{dfn}[thm]{Definition}
\newtheorem{rem}[thm]{Remark}

\newcommand{\R}{\mathbb{R}}
\newcommand{\C}{\mathbb{C}}
\newcommand{\eps}{\varepsilon}
\newcommand\HNN{\widehat{N}}
\newcommand{\N}{\mathbb{N}}
\newcommand{\NN}{\mathcal{N}}
\newcommand{\MM}{\mathcal{M}}
\newcommand{\MMR}{\mathcal{M}_{s}}
\newcommand{\PP}{\mathcal{P}}
\newcommand{\CC}{\mathcal{C}}
\newcommand{\CB}{\mathcal{C}_b}
\newcommand{\clc}{\overline{\mathrm{conv}}\,}

\title{A note on Rainwater's theorem}
\author{David Preiss}
\date{}

\begin{document}

\begin{abstract}
  We extend Rainwater's theorem that, for Baire measures on a compact Hausdorff space,
  a measure singular with respect to any measure from a $w^*$-compact convex set $M$ of probability measures sits on
  a Baire set which is $\mu$-null for every $\mu\in M$. Our main result shows that this statement holds
  under the weaker assumption of convex analyticity of $M$. In the standard way, this extends
  the Glicksberg--König--Seever decomposition theorem.
   We also give conditions for our results to hold for sets of signed measures.
\end{abstract}

\maketitle

The remarkable, and very useful (see, e.g.,~\cite{AM} for a recent application), result of Rainwater \cite{R} (most often referred to \cite[Lemma 9.4.3]{Ru}) gives conditions under which a measure
singular with respect to each measure from a certain collection $M$ sits on a set that is null for any measure from $M$. 
It concerns Baire probability measures on compact Hausdorff spaces and the main condition is that the collection of measures is convex
and compact in the $w^*$-topology. Although \cite{R} gives simple examples showing that
the convexity and $w^*$-compactness are ``vital'', we show that they can be usefully weakened. This leads also to a natural extension of the
Glicksberg--König--Seever decomposition of any measure as a sum of a measure absolutely continuous with respect a measure from $M$
and a measure singular to every measure from $M$.

Although not much would be lost by working in compact Hausdorff spaces, with possible applications in mind it is better to work
in an arbitrary topological space~$X$. For spaces of continuous functions we will use the notation $\CC(X,Y)$ for
functions with values in $Y$, $\CC(X)$ when $Y=\R$, and $\CC_b(X)$ for the space of bounded functions from $\CC(X)$.

For measure theoretical notions and results we refer to \cite{B}.
In particular, we recall that Baire sets are members of the smallest $\sigma$-algebra with respect to which
all continuous real-valued functions are measurable. It may be also described as the $\sigma$-algebra generated by
the zero-sets, i.e., sets of the form $f^{-1}(0)$ where $f\in\CC(X)$. To state some of our results in a more
precise form, we will denote by $Z_\sigma$ (zero-sigma) the collection of countable unions of zero-sets.

We denote by $\MM_s(X)\supset \MM(X) \supset \PP(X)$ the collections
of signed ($=$ real-valued), non-negative and probability Baire measures (defined on Baire sets in~$X$), respectively.
In addition to the standard notation for the integral of a function $h$ with respect to $\mu$ we will
often write $\mu(h)$. If $h$ is $\mu$-integrable, we denote by $h\mu$ the measure defined by
$(h\mu)(E):=\int_E h\,d\mu$.

We say that $\mu$ sits on a Baire set $B\subset X$ if $\mu(E)=0$ for every subset $E$ of $ X\setminus B$.
Measures $\mu$ and $\nu$ are said to be mutually singular, denoted $\mu\perp\nu$, if they sit on
disjoint Baire sets.
Every $\mu\in\MMR(X)$ can be uniquely expressed as $\mu=\mu^+-\mu^-$, where $\mu^+,\mu^-\in \MM(X)$ and $\mu^+\perp \mu^-$.
We also denote $|\mu|=\mu^++\mu^-$ and recall that for every non-negative $f\in\CB(X)$,
$|\mu|(f)=\sup\{\mu(g): g\in \CC(X),\, |g| \le f\}$.

The spaces of measures we work with are in a natural way subspaces of the dual to the space $\CB(X)$ equipped with the supremum norm.
From their norm topology we will however use only the notation
$\|\mu\|=|\mu|(1)$. Instead, we will always consider $\MMR(X)$ (and so also $\MM(X)$ and $\PP(X)$) equipped with the weak$^*$-topology,
i.e., the coarsest topology which makes all functions $\mu\to\mu(f)$, $f\in\CB(X)$, continuous.
Below, we will therefore skip the prefix $w^*$, even in Rainwater's theorem.

\section*{Rainwater's theorem}

Much of what we do here starts from the following Lemma, whose statement and proof 
are minor modifications of Rainwater's argument from \cite{R}.

\begin{lem}\label{rn}
  Let $M,N\subset\MMR(X)$ be compact convex sets such that
  $\mu\perp\nu$ whenever $\mu\in M$ and $\nu\in N$.
  Then for every $\eps>0$ there is $g\in \CC(X,[0,1])$ such that
  $|\mu(1-g)|<\eps$ for every $\mu\in M$ and $|\nu(g)|<\eps$ for every $\nu\in N$.
\end{lem}

\begin{proof}
  Replacing $M$ by $\{\mu_1-\mu_2: \mu_1,\mu_2\in M\}$ if necessary, we may assume that
  $-\mu\in M$ whenever $\mu\in M$. Similarly, we may assume that $-\nu\in N$ whenever $\nu\in N$.

  Define $f\colon (M\times N)\times \CC(X,[0,1])\to\R$ by
  $f((\mu,\nu),g):=\mu(1-g)+\nu(g)$.
  Then $M\times N$ is a compact convex subset of $\MM_s(X)\times\MMR(X)$
  and $f$ satisfies the assumptions of various extensions of the
  von Neumann minimax theorem (see, for example \cite[page 116]{G}, \cite[Theorem 2]{KF} or \cite[Corollary 3.3]{S}).
  Hence
  \begin{equation}\label{rn.E1}
    \adjustlimits\sup_{\mu,\nu}\inf_{g} \mu(1-g)+\nu(g) = \adjustlimits\inf_g\sup_{\mu,\nu} \mu(1-g)+\nu(g).
  \end{equation}
  Since the left side of \eqref{rn.E1} is less or equal to zero, so is its right side.
  Thus, given $\eps>0$ there is $g\in\CC(X,[0,1])$ such that
  $\mu(1-g)+\nu(g)<\eps$ for every $\mu\in M$ and $\nu\in N$. Choosing $\sigma,\tau\in\{-1,1\}$
  such that $\sigma\mu(1-g) = |\mu(1-g)|$ and $\tau\nu(g)=|\nu(g)|$ and recalling that
  $\sigma\mu\in M$ and $\tau\nu\in N$, we get  $|\mu(1-g)|+|\nu(g)|<\eps$, implying the statement.
\end{proof}

A straightforward consequence of \Cref{rn} is the following variant of
Rainwater's theorem.

\begin{thm}\label{r1}
  Let $M,N \subset\MM(X)$ be compact convex sets such that $\mu\perp\nu$ whenever $\mu\in M$ and $\nu\in N$.
  Then there are disjoint $Z_\sigma$ sets $B_M,B_N$ such that every $\mu\in M$ sits on $B_M$
  and every $\nu\in N$ sits on $B_N$.
\end{thm}

\begin{proof}
  Let $g_n\in \CC(X,[0,1])$ come from \Cref{rn} with $\eps_n:=4^{-n}$ and put $E_n:=\{x: g_n(x)\ge 1- 2^{-n}\}$
  and $F_n:=\{x: g_n(x)\le 2^{-n}\}$. Then $E_n$ and $F_n$ are zero-sets
  such that $\mu(X\setminus E_n)\le 2^n\mu(1-g_n)\le 2^{-n}$ for $\mu\in M$, and 
  $\nu(X\setminus F_n)\le 2^n\nu(g_n)\le 2^{-n}$ for $\nu\in N$.
  Hence it suffices to let
  \[B_M:=\bigcup_{k=1}^\infty \bigcap_{n=k}^\infty E_n\text{ and }
    B_N:=\bigcup_{k=1}^\infty \bigcap_{n=k}^\infty F_n.\qedhere\]
\end{proof}

Recall that in Rainwater's theorem $M$ consists of a single $\mu\in\MMR(X)$.
This is a straightforward special case of \Cref{r1} since
one may replace $\mu$ by $|\mu|$. It is however not clear whether the result
holds when $M$ is a set of signed measures. (Notice that the set $\{|\mu|:\mu\in M\}$ need be neither compact nor convex.)
We give a condition, called \ref{sm} for want of a better name,
under which this is the case.

\begin{enumerate}[leftmargin=0cm,labelwidth=1.5em,itemindent=\labelwidth,label=$(\star)$]
\item\label{sm}
 There is a countable set $H\subset \CC(X,[-1,1])$ such that for every $\mu\in M$,
  \[|\mu|(X)=\sup\{\mu(h): h\in H\}.\]  
\end{enumerate}
Particular cases when \ref{sm} holds are $M\subset\MM(X)$, $M$ is  a singleton, or any $M\subset \MMR(X)$ when $X$ is a separable metric space.
The first two cases are obvious, and to see the last one we take a countable basis $\mathcal B$ for the topology of $X$
which is closed under finite unions, for each $B_1,B_2\in \mathcal B$ with $\overline{B_1}\cap\overline{B_2}=\emptyset$ choose
$h_{B_1,B_2}\in\CC(X,[-1,1]) $ such that $h_{B_1,B_2}(x)=1$ for $x\in B_1$ and 
$h_{B_1,B_2}(x)=-1$ for $x\in B_2$, and let $H$ be the collection of the functions~$h_{B_1,B_2}$. 

A variant of Rainwater's theorem in which $M$ is allowed to be a set of signed measures is another
consequence of \Cref{rn}. Notice however, that the set on which measures from $M$
sit is no longer claimed to be $Z_\sigma$.

\begin{thm}\label{rl0}
  Let $M\subset\MM_s(X)$ satisfying \ref{sm} and $N \subset\MM(X)$ be compact convex sets
  such that $\mu\perp\nu$ whenever $\mu\in M$ and $\nu\in N$.
  Then there is a $Z_\sigma$ set $B$ such that every $\mu\in M$ sits on $X\setminus B$
  and every $\nu\in N$ sits on $B$.
\end{thm}

\begin{proof}
  Choose $h_1,h_2,\ldots\in \CC(X,[-1,1])$ with
  $\|\mu\| =\sup \{\mu(h_i): i\in\N\}$ for every $\mu\in M$,  such that in the
  sequence $h_i$, $i=1,2,\dots$ each element appears infinitely often.

  For each $i$ use \Cref{rn}
  with the sets $M_i:=\{h_i\mu: \nu\in N\}$ and $N$, and with $\eps=4^{-i}$,
  to find $g_{i}\in \CC(X,[0,1])$ such that
  $\mu(h_i(1-g_i))< 4^{-i}$ for every $\mu\in M$ and
  $\nu(g_i)<4^{-i}$ for every $\nu\in N$.
  Put
  \[D_i:=\{x\in X: g_{i}(x) \le 2^{-i}\},\
    B_j:=\bigcap_{i=j}^\infty D_{i}\text{ and } B:=\bigcup_{j=1}^\infty B_j.\]

  For any $\nu\in N$ we get
  $\nu(X\setminus D_i)\le 2^{i}\nu(g_i) \le 2^{-i}$, hence $\nu(X\setminus B_j)\le 2^{-j+1}$,
  and $\nu(X\setminus B)=0$. Hence $\nu$ sits on $B$.

  It remains to show that every $\mu\in M$ sits on $X\setminus B$, i.e, that $|\mu|(B)=0$.
  For this, given $j\in \N$, find $i>j$ such that $\|\mu\|-\mu(h_i)< 4^{-j-1}$. Since
  $|h_ig_i|\le g_i$, $\mu(h_ig_i)\le |\mu|(g)$, and hence
  \[|\mu|(1-g_i)=\|\mu\|-|\mu|(g_i)\le 4^{-i}+\mu(h_i)-\mu(h_ig_i)\le 2\star4^{-j-1}.\]
  Since $1-g_i\ge 1/2$ on $D_i$, we see that $|\mu|(B_j)\le |\mu|(D_i)\le 4^{-j}$
  for every~$j$. Since $B_j\subset B_{j+1}$, this shows
   $|\mu|(B_j)=0$ for every~$j$ and so $|\mu|(B)=0$.
\end{proof}

\section*{Convexly analytic sets of measures}

The main aim of this note is to replace the assumption on $N$ in Rainwater's theorem
(\Cref{r1} where $M$ consists of a single measure)
by its convex analyticity, which is a notion introduced by Petr Holický in~\cite{H}.

We recall several basic notions and facts
from the descriptive set theory.
A multi-valued map $\phi: S\to T$
between topological spaces $S$ and $T$ is called usco (upper semicontinuous compact valued)
if its values are compact subsets of $T$ and for every open set $G\subset T$ the set
$\{s\in S: \phi(s)\subset G\}$ is open. The $\phi$-image of $U\subset S$ is defined by
$\phi(U):=\bigcup_{u\in U} \phi(u)$. We notice that $\phi(U)$ is compact when $U$ is compact.
A set $U\subset S$ is called analytic if it is an usco image of a polish space (metrizable by a complete separable metric).
Letting $\NN$ be the space of sequences of positive integers with the product topology and
recalling that every nonempty polish space is a continuous image of $\NN$,
we see that $U$ is analytic if and only if it is an usco image of $\NN$.

We will use the following standard notation for certain subsets of $\NN$:
\begin{align*}
  \NN_{\sigma_1,\dots,\sigma_k}&:=\{\tau\in\NN: \tau_i\le \sigma_i\text{ for $i=1,\dots,k$}\}
\text{ for $\sigma_1,\dots,\sigma_k\in\N$,}\\  
  \shortintertext{and}
  \NN_\sigma&:=\{\tau\in\NN: \tau_i\le \sigma_i\text{ for $i=1,2,\dots$}\} \text{ for $\sigma\in\NN$.}
\end{align*}

We will also use the observations that
the sets $\NN_\sigma$ are compact, 
every compact subset of $\NN$ is contained in some $\NN_\sigma$ and,
when $\phi:\NN\to T$ is usco, then $\psi(\sigma):=\phi(\NN_\sigma)$
is also usco and $\psi(\NN)=\phi(\NN)$.

\begin{dfn}\label{ca0}
  A subset $U$ of a locally convex space $T$ is said to be convexly analytic if
  there is an usco mapping 
  $\phi:\NN\to T$
  such that $U=\phi(\NN)$ and the closed convex hull of $\phi(K)$ is contained in $U$ for every compact $K\subset\NN$.
\end{dfn}

Recalling again that every nonempty polish space is a continuous image of~$\NN$, we see
that a set $U\subset T$ is convexly analytic provided it is an image of an usco mapping
$\phi$ of a polish space $P$ such that
the closed convex hull of $\phi(K)$ is contained in $U$ for every compact $K\subset P$.

Since we will use convex analyticity only in the space $\MMR(X)$, we recall the relation
of convexity and the notion of barycenter only in this special case.
When $\lambda\in\PP(\MMR(X))$ and $\int \|\mu\|\,d\lambda(\mu)<\infty$, its \emph{barycenter}
is the measure $\hat\lambda\in\MMR(X)$ defined by 
$\hat\lambda(g):=\int \mu(g)\,d\lambda(\mu)$ for $g\in \CB(X)$.
If $M\subset\MMR(X)$ is compact, the map $\lambda\in\PP(M)\to\hat\lambda\in\MMR(X)$ is continuous.
Hence the set $\{\hat\lambda: \lambda\in \PP(M)\}$ is a compact and convex
subset of $\MM_s(X)$ containing $M$, and so also $\clc M$. Consequently, $\clc M$ is compact.
Moreover, for any $g\in\CB(X)$ and $\lambda\in\PP(M)$ we have
$\hat\lambda(g)\le\max_{\mu\in M} \mu(g)$, and so the Hahn-Banach theorem implies that
$\clc M = \{\hat\lambda: \lambda\in \PP(M)\}$.

Although the following simple fact showing a way of naturally enlarging an analytic set to a convexly analytic set
clearly holds more generally, we give it only in its special case that will be used in the proof of \Cref{T-4}.

\begin{obs}\label{a1}
   Suppose $\phi$ is an usco map of $\NN$ to $\MM_s(X)$.
   Then the map $\psi$ assigning to $\sigma\in\NN$ the closed convex hull of $\phi(\NN_\sigma)$ is usco.
\end{obs}

\begin{proof}
  From the above discussion we see that $\psi$ is compact valued. Hence we just need to show that it is upper semicontinuous.
  For that, assume $\sigma\in\NN$ and $G\supset \psi(\sigma)$ is open. Since $\psi(\sigma)$ is compact,
  there is an open convex set $V\ni 0$ such that $\overline{V+\psi(\sigma)}\subset G$. Then $V+\psi(\sigma)$ is an open set
  containing the compact set $\phi(\NN_\sigma)$, hence there is $k$ such that
  $\phi(\NN_{\sigma_1,\dots,\sigma_k})\subset V+\psi(\sigma)$.
  Since $\overline{V+\psi(\sigma)}$ is closed and convex,
  we get $\psi(\tau)\subset \overline{V+\psi(\sigma)}\subset G$ for every $\tau\in\NN_{\sigma_1,\dots,\sigma_k}$.
\end{proof}

\section*{Rainwater's theorem for convexly analytic sets of measures}

Here, in \Cref{T1}, we prove the main result of this note that replaces the compactness and convexity assumptions
of Rainwater's theorem by convex analyticity. Its proof uses one of the basic arguments to show
separation results for analytic sets.

\begin{dfn}\label{defsep}
  A measure $\mu\in\MMR(X)$ is said to be $\eps$-separated from a set $N\subset\MMR(X)$, where $\eps\ge 0$, if there is
  a Baire measurable function $f\colon X\to[0,1]$ such that
  $|\mu|(1-f)\le\eps$ and $|\nu|(f)\le\eps+2\|\nu^-\|$ for every $\nu\in N$.
\end{dfn}

\begin{lem}\label{s1}
  Let $N_1\subset N_2\subset\dots\subset\MMR(X)$,
  $\eps_1,\eps_2,\ldots\ge 0$, $\mu\in\MM_s(X)$ be $\eps_i$\nobreakdash-separated from $N_i$ for each $i$,
  and $\eps:=\limsup_{i\to\infty} \eps_i$. Then
  $\mu$ is $\eps$\nobreakdash-separated from 
  $\bigcup_{i=1}^\infty N_i$.
\end{lem}
	
\begin{proof}
  Let $\mu$ be $\eps_i$-separated from $N_i$ by a function $f_i$. Considered in $L_2(|\mu|)$,
  the sequence $f_i$ contains a weakly convergent subsequence and so by
  Mazur's Lemma there are norm convergent convex combinations
  $g_j:=\sum_{i=j}^\infty c_{i,j} f_i$. Hence there is a
  subsequence $g_{j_k}$ of $g_j$ which converges $\mu$ almost everywhere.
  Let $E$ be the (necessarily Baire) set of $x$ for which $g_{j_k}(x)$ converges. Define
  $h_k(x)=g_{j_k}(x)$ for $x\in E$ and 
  $h_k(x)=0$ otherwise, and put $f(x):=\lim_{k\to\infty} h_k(x)$.
  Then $f$ is a well defined Baire measurable map of $X$ to $[0,1]$,
  \begin{align*}
    |\mu|(1-f)&=
              \lim_{k\to\infty} |\mu|(1-g_{j_k})
              \le \limsup_{i\to\infty} |\mu|(1-f_i) \le 
              \limsup_{i\to\infty}\eps_i=\eps\\
    \intertext{and for every $\nu\in M$,}
    |\nu|(f)&=\lim_{k\to\infty} |\nu|(h_k)\le \limsup_{j\to\infty}|\nu|(g_j)\le
      \limsup_{i\to\infty} |\nu|(f_i)\\
           &\le \limsup_{i\to\infty}\eps_i +2\|\nu^-\| =\eps+2\|\nu^-\|.\qedhere
  \end{align*}
\end{proof}

\begin{cor}\label{T-0}
  Suppose $N\subset\MMR(X)$ is convexly analytic and $\mu\in\MMR(X)$ is such that
  $\mu\perp\nu$ for every $\nu\in N$.
  Then $\mu$ sits on a Baire set $B\subset X$ such that
  $|\nu|(B)\le 2\|\nu^-\|$
  for every $\nu\in N$.
\end{cor}

\begin{proof}
  Our goal is to show that for every $\eps>0$, $\mu$ is $\eps$-separated from $N$. 
  Then \Cref{s1} with $N_i:=N$ and $\eps_i:=1/i$ implies that
  $\mu$ is $0$-separated from~$N$. Hence there is a Baire measurable function $f\in\CC(X,[0,1])$
  such that $|\mu|(1-f)=0$ and $|\nu|(f)\le 2\nu^-(f)$, and we see that
  the statement holds with $B:=\{x:f(x)= 1\}$.

  Suppose, for a contradiction, that for some $\eps>0$, $\mu$ is not $\eps$-separated from~$N$. By \Cref{ca0}, $N=\phi(\NN)$,
  where $\phi:\NN\to\MMR(X)$ is usco and such that for every $\sigma\in\NN$ the set $C_\sigma:=\clc\phi(\NN_\sigma)$ is
  a convex subset of~$N$.
  Since $\phi(\NN_1)\subset\phi(\NN_2)\subset\dots$ and
  $N=\bigcup_{\sigma_1=1}^\infty \phi(\NN_{\sigma_1})$, by \Cref{s1} there is $\sigma_1\in\N$ such that $\mu$ is not $\eps$-separated from
  $\phi(\NN_{\sigma_1})$. Continuing recursively, we find $\sigma_k\in\N$ such that $\mu$ is not $\eps$-separated
  from $\phi(\NN_{\sigma_1,\dots,\sigma_k})$ for each $k=1,2,\dots$.
  Letting $\sigma:=(\sigma_1,\sigma_2,\dots)$ and recalling that $C_\sigma:=\clc\phi(\NN_\sigma)$ is
  a compact and convex subset of $N$,
  we use \Cref{rn} with $M:=\{|\mu|\}$ to find $f\in C(X,[0,1])$
  such that $|\mu|(1-f)<\eps$ and $\nu(f)<\eps$ for every $\nu\in C_\sigma$.
  Since the map $\nu\in\MMR(X)\to\nu(f)$ is continuous,
  $G:=\{\nu\in\MMR: \nu(f)<\eps\}$ is an open set containing $\phi(\NN_\sigma)$.
  Hence the upper semicontinuity of $\phi$ implies 
  that there is $k$ such that $\nu \in G$ for every $\nu\in \phi(\NN_{\sigma_1,\dots,\sigma_k})$. 
  It follows that $|\nu|(f)=\nu(f)+2\nu^-(f)\le\eps+2\|\nu^-\|$ for every such $\nu$, 
  and we conclude that $\mu$
  is $\eps$-separated from $\phi(\NN_{\sigma_1,\dots,\sigma_k})$, which is the desired contradiction.
\end{proof}

\begin{cor}\label{T-1}
  Suppose $N\subset\MMR(X)$ is convexly analytic, $\mu\in\MMR(X)$ is such that
  $\mu\perp\nu$ for every $\nu\in N$ and $h\in\CC(X,[-1,1])$.
  Then $\mu$ sits on a Baire set $B\subset X$ such that
  $|\nu|(B)\le 2(\|\nu\|-\nu(h))$
  for every $\nu\in N$.
\end{cor}

\begin{proof}
  By \Cref{T-0}, since $\{h\nu: \nu\in N\}$ is convexly analytic, $\mu$ sits on a Baire set $B$
  such that $|h\nu|(B)\le 2\|(h\nu)^-\|=\|h\nu\|-\nu(h)\le \|\nu\|-\nu(h)$.
  Hence
  \begin{align*}
    |\nu|(B)&= |\nu|(X)-|\nu|(X\setminus B)\le |\nu|(X)-(h\nu)(X\setminus B)\\
            & \le \|\nu\|-\nu(h)+|h\nu|(B)\le 2(\|\nu\|-\nu(h)).\qedhere
  \end{align*}
 \end{proof}

\begin{cor}\label{T-2}
  Suppose $N\subset\MMR(X)$ is convexly analytic, $\mu\in\MMR(X)$ is such that
  $\mu\perp\nu$ for every $\nu\in N$ and $H\subset\CC(X,[-1,1])$ is countable.
  Then $\mu$ sits on a Baire set $B\subset X$ such that
  $|\nu|(B)\le 2\inf \{\|\nu\|-\nu(h): h\in H\}$
  for every $\nu\in N$.
\end{cor}

\begin{proof}
  By \Cref{T-1}, for every $h\in H$ there is a a Baire set $B_h\subset X$ such that
  $\mu$ sits on $B_h$ and
  $|\nu|(B_h)\le \|\nu\|-\nu(h)$
  for every $\nu\in N$. Since $H$ is countable, the set $B:=\bigcap_{h\in H} B_h$ has
  the stated properties.
\end{proof}

A straightforward application of \Cref{T-2} gives an extension of Rainwater's theorem to convexly analytic sets.
\begin{thm}\label{T1}
  Suppose a convexly analytic set $N\subset\MMR(X)$ satisfies \ref{sm} and $\mu\in\MMR(X)$ is such that
  $\mu\perp\nu$ for every $\nu\in N$.
  Then $\mu$ sits on a Baire set $B\subset X$ such that every $\nu\in N$ sits on $X\setminus B$.
\end{thm}

Finally, we prove a variant of the Glicksberg--König--Seever decomposition theorem (\cite{G,KS} or \cite[9.4.4]{Ru}).

\begin{thm}\label{T-4}
  Let $N\subset\MMR(X)$ be an analytic set satisfying \ref{sm}. Then for every $\mu\in\MM_s(X)$
  there are 
  a Baire function $g:X\to\R$, a Baire set $B\subset X$ and $\lambda\in\PP(N)$ such that $\int\|\nu\|\,d\lambda(\nu)<\infty$,
  $g(x)=0$ for $x\in B$, $\int|g(x)|\,d|\hat\lambda|(x)<\infty$,
  $|\nu|(B)=0$ for every $\nu\in N$, and for every Baire set $E\subset X$, 
  \begin{equation}
    \label{E}\mu(E)=\mu(E\cap B)+ \int_{E\setminus B} g(x)\,d\hat\lambda(x)
  \end{equation}
\end{thm}

\begin{proof}
  Choose an usco map $\phi:\NN\to X$ with $\phi(\NN)=N$ put
  \[\hat N:= \{\hat\lambda: \lambda\in \PP(\phi(\NN_\sigma))\text{ for some $\sigma\in\NN$}\}\]
  and recall that $\hat N$ is convexly analytic by \Cref{a1}.
  
  Let $\mathcal{U}$ be the collection of Baire sets $U\subset X$ for which there is $\hat\nu\in\HNN$
  such that $\hat\nu(V)\ne 0$ for every Baire set $V\subset U$ with $\mu(V)\ne 0$. Choose $U_i\in\mathcal{U}$
  for which $|\mu|(U)$, where $U:=\bigcup_i U_i$, is largest possible and let $\hat\nu_i\in\HNN$ be the corresponding measures.
  
  Define $\hat\mu\in\MMR(X)$ by $\hat\mu(E):=\mu(E\setminus U)$.
  Then $\hat\mu\perp\hat\nu$ for every $\hat\nu\in\HNN$. Hence by \Cref{T-2} there is a Baire set $B\subset X\setminus U$ such that
  $|\hat\mu|(X\setminus B)=0$ and $|\nu|(B)=0$ for every $\nu\in N$,
  hence also $|\hat\nu|(B)=0$ for every $\hat\nu\in\hat N$.

  For each $i$ there is a choice $g_i$ of the Radon-Nikodym derivative of the absolutely continuous part of $\hat\nu_i$ with respect to $\mu$
  such that $g_i(x)=0$ for $x\in B$.
  By definition of $\HNN$,
  $\hat\nu_i$ is a barycenter of some $\lambda_i\in\PP(N)$ with compact support.
  Since
  $\|\nu\|$ is bounded on the support of $\lambda_i$, there are $\kappa_i>0$ such that
  $\kappa_i\int\|\nu\|\,d\lambda_i(\nu)\le 2^{-i}$ and $\kappa_i\int|g_i(x)|\,d|\mu|(x)<2^{-i}$.

  Put $V:=\{x\in U: \sum_i \kappa_i |g_i(x)| <\infty\}$. Then $\mu(U\setminus V)=0$, and hence
  $h_i(x):=g_i(x)$ for $x\in V$ and $h_i(x):=0$ for $x\notin V$ is also
  the Radon-Nikodym derivative of the absolutely continuous part of $\hat\nu_i$ with respect to $\mu$.
  
  Given $t\in[0,1]^{\N}$ we let $\lambda_t:=\sum_i\kappa_i t_i \lambda_i$ and observe that
  $\int\|\nu\|\,d\lambda_t(\nu)\le 1$, the barycenter of $\lambda_t$ is $\nu_t:=\sum_i\kappa_i t_i \nu_i$,
  and the Radon-Nikodym derivative of the absolutely continuous part of $\nu_t$ with respect to $\mu$ is
  the function $g_t$ defined by $g_t(x):=\sum_i\kappa_i t_i g_i(x)$ for $x\in V$ and $g_t(x)=0$ otherwise.
  When $x\in V\cap U_j$ and $t_i$, $i\ne j$, are fixed, there is at most one value
  of $t_j$ such that $g_t(x)=0$. Hence $\Lambda\{t: g_t(x)=0\}$, where $\Lambda$ is the product Lebesgue measure on $[0,1]^\N$,
  and so $\mu\{x\in V\cap U_j: g_t(x)=0\}=0$ for $\Lambda$ almost every $t$ by the Fubini theorem.
  It follows that there is $t\in (0,1]^\N$ such that $g_t(x)\ne 0$ for $\mu$ almost every $x\in V$.
  Fix such a $t$, choose $c>0$ such that $\lambda:=c\lambda_t \in\PP(N)$ and put
  \[h(x):=
    \begin{cases}
      1 &\text{when $x\in V$ and $g_t(x)=0$,}\\
       h(x):=cg_t(x)& \text{for all other $x$.}
    \end{cases}
  \]    

  Let $W:=\{x\in V: g_t(x)\ne 0\}$. The restrictions of $\hat\lambda$ and $\mu$ to $W$ are mutually absolutely continuous.
  Since for every $E\subset W$, $\hat\lambda(E)=\int_E cg_t(x)\,d\mu(x)$, we get
  $\mu(E)=\int_E g(x)\,d\hat\lambda(x)$ where $g(x):=1/(cg_t(x))$ for $x\in W$. Letting $g(x)=0$ otherwise, and
  writing $|\hat\lambda|=h\lambda$ where $|h|=1$, we get
  \[\int_X |g(x)|\,d|\hat\lambda|(x)=\int_W h(x)|g(x)|/g(x)\,d\mu(x)<\infty,\]
  and, since $|\hat\lambda|(X\setminus U)=0$ and $|\mu|(U\setminus W)=0$, we see that \eqref{E} holds.
\end{proof}

\begin{cor}\label{T-5}
  Let $N\subset\MM(X)$ be an analytic set of measures on a complete separable metric space.
  Then for every $\mu\in\MM(X)$
  there are 
  a Borel function $g:X\to[0,\infty)$, a Borel set $B\subset X$ and $\lambda\in\PP(N)$ such that 
  $g(x)=0$ for $x\in B$,
  $\nu(B)=0$ for every $\nu\in N$ and 
    \[\mu(E)=\mu(E\cap B) + \int_N \int_E g(x)\,d\nu(x)\,d\lambda(\nu)\]
  for every Borel set $E\subset X$.
\end{cor}

\begin{proof}
  As already observed above, for complete separable metric spaces the condition \ref{sm} is automatically satisfied.
  Hence we simply replace the $g$ from \Cref{T-4} by $\max(0,g)$, which is justified 
  by observing that \eqref{E} applied to any set $D\subset\{x: g(x)<0\}$ gives
  $\int_N \int_D g(x)\,d\nu(x)\,d\lambda(\nu)=0$.
\end{proof}

\begin{rem}
  Since for every $\mu\in\MMR(X)$ and a Baire set $B$ there is a $Z_\sigma$ set $Z\subset B$ such that $|\mu|(B\setminus Z)=0$, 
  the set $B$ in each of the statements of \Cref{T-0} -- \Cref{T-5} can be taken to be $Z_\sigma$.
\end{rem}

\medskip
\noindent\footnotesize
Mathematics Institute, University of Warwick, Coventry CV4 7AL, United Kingdom.

\noindent Email: d.preiss@warwick.ac.uk

\end{document}